\RequirePackage{fix-cm}
\documentclass[smallextended]{svjour3}       
\smartqed  
\usepackage{graphicx}
\usepackage{amsmath,amssymb}
\usepackage{color}
\usepackage{multirow}
\usepackage{booktabs}
\usepackage{url}
\begin{document}

\title{New semidefinite relaxations for a class of complex quadratic programming problems\thanks{This preprint has not undergone peer review (when applicable) or any post-submission improvements or corrections. The Version of Record of this article is published in Journal of Global Optimization, and is available online at https://doi.org/10.1007/s10898-023-01290-z.}
}

\titlerunning{New semidefinite relaxations for complex quadratic program}        

\author{Yingzhe Xu \and
        Cheng Lu \and
        Zhibin Deng \and
        Ya-Feng Liu
}


\institute{Yingzhe Xu \at
              School of Economics and Management, North China Electric Power University, Beijing 102206, China\\
              \email{xuyingzhe1@163.com}
           \and
           Cheng Lu \at
              School of Economics and Management, North China Electric Power University, Beijing 102206, China\\
              \email{lucheng1983@163.com}
           \and
           Zhibin Deng \at
              School of Economics and Management, University of Chinese Academy of Sciences; MOE Social Science Laboratory of Digital Economic Forecasts and Policy Simulation at UCAS, Beijing, 100190, China\\
              \email{zhibindeng@ucas.ac.cn}
           \and
           Ya-Feng Liu \at
              State Key Laboratory of Scientific and Engineering Computing, Institute of Computational Mathematics and Scientific/Engineering Computing, Academy of Mathematics and
              Systems Science, Chinese Academy of Sciences, Beijing, 100190, China\\
              \email{yafliu@lsec.cc.ac.cn}
}

\date{Received: date / Accepted: date}

\maketitle

\begin{abstract}
In this paper, we propose some new semidefinite relaxations for a class of nonconvex complex quadratic programming problems, which widely appear in the areas of signal processing and power system. By deriving new valid constraints to the matrix variables in the lifted space, we derive some enhanced semidefinite relaxations of the complex quadratic programming problems. Then, we compare the proposed semidefinite relaxations with existing ones, and show that the newly proposed semidefinite relaxations could be strictly tighter than the previous ones. Moreover, the proposed semidefinite relaxations can be applied to more general cases of complex quadratic programming problems, whereas the previous ones are only designed for special cases. Numerical results indicate that the proposed semidefinite relaxations not only provide tighter relaxation bounds, but also improve some existing approximation algorithms by finding better sub-optimal solutions.
\keywords{Quadratic optimization \and Semidefinite relaxation \and Approximation algorithm \and Phase quantized waveform design \and Discrete transmit beamforming}
\end{abstract}

\section{Introduction}
In this paper, we consider the following complex quadratic programming problem:
\begin{equation}\label{CQP}\tag{CQP}
\begin{aligned}
\min~&x^{\dag} Q_{0} x\\
\textrm{s.t.} ~&x^{\dag} Q_{i} x \leq b_{i},~i=1, \ldots, m, \\
&l_i \leq {|}x_i {|} \leq u_i,~ i=1, \ldots, n, \\
&\arg(x_i{x_j}^{\dag}) \in \mathcal{A}_{ij}, \{i,j\}\in \mathcal{E},
\end{aligned}
\end{equation}
where $Q_0,Q_1,\ldots,Q_m\in \mathbb{C}^{n\times n}$ are Hermitian matrices, $b_1,\ldots,b_m\in \mathbb{R}$, and $l_i$ and $u_i$ are the lower and upper bounds of the modulus of $x_i\in \mathbb{C}$ for $i=1,\ldots,n$. The set $\mathcal{E}$ is a subset of $ \{\{i,j\}\,|\,1\leq i<j\leq n\}$. For each $\{i,j\}\in \mathcal{E}$, $\mathcal{A}_{ij}$ is either an interval of the form $[\underline{\theta}_{ij},\bar{\theta}_{ij}]$, or a set of discrete points of the form $\{\theta_{ij}^1,\ldots,\theta_{ij}^M\}\subseteq \mathbb{R}$. We use $\arg(\cdot)$ to denote the argument of a complex variable, and $(\cdot)^\dag$ to denote the conjugate transpose of a matrix/vector.

Problem \eqref{CQP} arises in many important applications in signal processing, communication and power system. For example, the radar waveform design problem \cite{Maio2011}, the transmit beamforming problem \cite{Demir2014,Demir2015}, and the optimal power flow problem \cite{Low2014a,Low2014b}, can all be formulated as special cases of \eqref{CQP}. Besides, when $l_i=u_i=1$ for $i=1,\ldots,n$, \eqref{CQP} degenerates to the unit-modulus constrained quadratic programming problem, which arises in applications including the MIMO detection problem \cite{Ma2004}, the radar phase code design problem \cite{Maio2009}, the angular synchronization problem \cite{Bandeira}, and the max-3-cut problem \cite{Goemans}. In some applications arising in network problems, such as the optimal power flow problem in the electricity network \cite{Chen2016,Coffrin}, the set $\mathcal{E}$ can be seen as an edge set that represents the network structure. Besides, the constraint $\arg(x_i{x_j}^{\dag}) \in \mathcal{A}_{ij}$ in \eqref{CQP}, which bounds the difference of phase angles, has some physical meanings in power systems and other related applications. One may refer to \cite{Coffrin} for further discussions on the physical meaning of this constraint.

Problem \eqref{CQP} is NP-hard in general, since some of its special cases are already known to be NP-hard \cite{Zhang2006}. Hence, it is not possible to solve \eqref{CQP} globally in polynomial-time complexity, unless P=NP. Instead, we are interested in designing efficient algorithms to find sub-optimal solutions of \eqref{CQP}. In the literature, most existing sub-optimal algorithms for solving \eqref{CQP} or its subclass problems are approximation algorithms, local-optimization algorithms, or heuristic algorithms \cite{Demir2014,Demir2015,Goemans,Maio2009,Maio2011,So2007,Waldspurger,Zhang2006,Zhao}. Among different sub-optimal algorithms, the semidefinite relaxation based approximation algorithms have attracted great attention since the pioneering work of Goemans and Williamson \cite{Goemans1995}. One may refer to \cite{Luo2010} for a comprehensive survey on the applications of semidefinite relaxation in signal processing, and \cite{Low2014a,Low2014b} for the applications in power system. For certain type of \eqref{CQP} which arises in real applications, it has been proven that under some practical conditions, the probability of obtaining the global solution of the problem using the semidefinite relaxation based approximation algorithms can be very high \cite{Bandeira,Low2014b,Lu2019}.

However, 
when the phase angle constraints $\arg(x_i{x_j}^{\dag}) \in \mathcal{A}_{ij}$ are presented in \eqref{CQP}, the current existing semidefinite relaxations are generally not tight, especially when $\mathcal{A}_{ij}$ is a discrete set. The reason is that these phase angle constraints are usually dropped when deriving semidefinite relaxations of \eqref{CQP}. Indeed, by exploiting the phase angle constraints, we may derive some new valid constraints to enhance the semidefinite relaxation. A theoretical result has been presented in \cite{Lu2019}, which shows that by introducing new valid constraints derived from the phase angle constraints, we may design an improved semidefinite relaxation for the MIMO detection problem. The tightness of the improved semidefinite relaxation can be guaranteed under certain conditions. Similar enhanced semidefinite relaxations have also been designed for other types of complex quadratic programming problems that appear in signal processing \cite{Lu2018,Lu2020}.

In \cite{Lu2018} and \cite{Lu2020}, Lu et al. have proposed a method of representing a complex variable by its polar coordinate form to derive tight semidefinite relaxations. The main idea behind the method is that some valid constraints can be easily derived in terms of polar coordinate variables, while it is hard to discover them in the complex coordinate variables. Based on this method, Lu et al. have proposed some improved semidefinite relaxations for several classes of complex quadratic optimization problems in \cite{Lu2018,Lu2020}. However, for problem \eqref{CQP} that contains the constraints $\arg(x_i{x_j}^{\dag}) \in \mathcal{A}_{ij}$, the previous semidefinite relaxations proposed in \cite{Lu2018,Lu2020} are not always applicable.

In this paper, we propose a new semidefinite relaxation, which is tighter and more flexible than the existing ones in \cite{Lu2018,Lu2020}. Our main idea is to lift the variable $x$ to a matrix $X=xx^\dag$, and exploit the valid inequalities on the matrix entries $X_{ii}$, $X_{jj}$ and $X_{ij}$ to enhance the tightness of a semidefinite relaxation. The contributions of this paper are two folds:
\begin{itemize}
\item From the theoretical aspect, we exactly describe the convex hull of the set $$\mathcal{J}^x_{ij}:=\{(x_i x^\dag_i,x_j x^\dag_j,x_i x^\dag_j)\,|\, ~l_i \leq {|}x_i {|} \leq u_i,l_j \leq {|}x_j {|} \leq u_j,\arg(x_i{x_j}^{\dag}) \in \mathcal{A}_{ij} \}.$$ It turns out that the convex hull of $\mathcal{J}^x_{ij}$ can be represented as a finite number of linear constraints and no more than two second-order cone constraints. Our theoretical result can be applied to general cases of \eqref{CQP}, including the case of $\mathcal{A}_{ij}$ being an interval, and the case of being a finite set. This new result generalizes the previous one in \cite{Chen2016}, in which the convex hull of $\mathcal{J}^x_{ij}$ is described only for the case $\mathcal{A}_{ij}$ being an interval contained in  $(-\pi/2,\pi/2)$.
\item From the computational aspect, based on the formulation of the convex hull of $\mathcal{J}^x_{ij}$, we introduce new valid inequalities on variables $X_{ii}$, $X_{jj}$ and $X_{ij}$ to design some enhanced semidefinite relaxations. Our results show that the enhanced semidefinite relaxations can be tighter than the existing semidefinite relaxations in the literature. Moreover, by adopting the enhanced semidefinite relaxations, some previous approximation algorithms can be improved to find better sub-optimal solutions.
\end{itemize}

The remaining parts of this paper are organized as follows: Section 2 introduces the first enhanced semidefinite relaxation, and analyzes some basic properties of the proposed semidefinite relaxation. Section 3 compares the proposed enhanced semidefinite relaxation with existing ones in the literature. Section 4 proposes the second enhanced semidefinite relaxation which is further enhanced from the one proposed in Section 2, and shows that the new semidefinite relaxation can be strictly tighter than the one proposed in \cite{Chen2016}. Section 5 presents the numerical results to show the performance of the proposed semidefinite relaxations.

The following notations will be adopted throughout the paper: For a given matrix $X\in \mathbb{C}^{n\times n}$, $\texttt{Re}(X)$
and $\texttt{Im}(X)$ denote its component-wise real part and imaginary part, respectively. For a given Hermitian matrix
$A\in \mathbb{C}^{n\times n}$, $A\succeq 0$ means $A$ is positive semidefinite. For two given Hermitian matrices $A$ and $B$, $A\succeq B$
means $A-B\succeq 0$. Moreover, $\texttt{Trace}(A)$ denotes the trace of $A$, $\texttt{rank}(A)$ denotes the rank of $A$, and $A\cdot B$ denotes $\texttt{Trace}(A^{\dag}B)$.
For a set $\mathcal{S}$ in some vector space, we use $\textsf{Conv} (\mathcal{S})$ to represent the convex hull of $\mathcal{S}$.
We use $$\mathcal{A}^M:=\left\{0,\frac{1}{M}2\pi,\ldots,\frac{M-1}{M}2\pi\right\}$$
to represent the uniformly discretized phase angle set for $M\geq 2$. Besides, with a slight abuse of notations, for a nonzero complex variable $z$ and a set $\mathcal{A}\subset \mathbb{R}$,
the notation $\arg(z)\in \mathcal{A}$ means that there exists a $k\in \mathbb{Z}$ such that $\arg(z)+2k\pi$ is contained in $\mathcal{A}$.

\section{A new semidefinite relaxation}
In this section, we first present a classical semidefinite relaxation for problem \eqref{CQP}, and then enhance it by introducing some new valid constraints. Letting $X=xx^{\dag}$, problem
\eqref{CQP} can be reformulated as follows:
\begin{equation}\tag{CQP2}\label{CQP2}
\begin{aligned}
\min~& Q_{0}\cdot X\\
\textrm{s.t.} ~&Q_{i}\cdot X \leq b_{i},~ i=1, \ldots, m, \\
&l_i^2 \leq  X_{ii} \leq u^2_i,~  i=1, \ldots, n, \\
&\arg(X_{ij}) \in \mathcal{A}_{ij}, \{i,j\}\in \mathcal{E},\\
&\texttt{rank}(X)=1.
\end{aligned}
\end{equation}
Dropping the constraints $\arg(X_{ij}) \in \mathcal{A}_{ij}$ for $\{i,j\}\in \mathcal{E}$ and relaxing $\texttt{rank}(X)=1$ to $X\succeq 0$, we have the following classical semidefinite
relaxation:
\begin{equation}\tag{CSDP}\label{CSDP}
\begin{aligned}
\min~& Q_{0}\cdot X\\
\textrm{s.t.} ~&Q_{i}\cdot X \leq b_{i},~ i=1, \ldots, m, \\
&l_i^2 \leq  X_{ii} \leq u^2_i, ~i=1, \ldots, n, \\
&X\succeq 0.
\end{aligned}
\end{equation}
\eqref{CSDP} has been widely used to design approximation algorithms in the literature. For example, Maio et al. have applied \eqref{CSDP} to design an approximation algorithm for solving the radar waveform design problem \cite{Maio2011}. Zhang and Huang \cite{Zhang2006}, and So et al. \cite{So2007}, have applied \eqref{CSDP} in the special case, where $l_i=u_i=1$ for $i=1,\ldots,n$, to design approximation algorithms for the unit-modulus constrained complex quadratic optimization.

However, since the constraints $\arg(X_{ij}) \in \mathcal{A}_{ij}$ are dropped directly, the bound provided by \eqref{CSDP} may not be tight enough for some classes of \eqref{CQP}. As discussed in \cite{Lu2018,Lu2019,Lu2020}, by exploiting the structure of the phase angle constraints, we may derive new valid inequalities to enhance the semidefinite relaxation.

For this purpose, we introduce the polar coordinate form for each complex variable $x_i=r_i e^{\texttt{i}\theta_i}$.
Then, we introduce two lifted matrices, including $X=xx^{\dag}\in \mathbb{C}^{n\times n}$, and $R=rr^{\dag}\in \mathbb{R}^{n\times n}$. For each $\{i,j\}\in \mathcal{E}$, we
have the following equations:
\begin{equation}\label{eq1a}
R_{ij}=r_i r_j,~X_{ij}=r_i r_j e^{\texttt{i}(\theta_i-\theta_j)}=R_{ij}e^{\texttt{i}(\theta_i-\theta_j)}.
\end{equation}
Based on \eqref{eq1a}, we may further derive the following equations:
\begin{equation}\label{eq5}
R_{ii}R_{jj}=R^2_{ij},~|X_{ij}|=R_{ij}.
\end{equation}
Adding \eqref{eq5} into \eqref{CSDP}, we have the following problem:
\begin{equation}\tag{CQP3}\label{CQP3}
\begin{aligned}
\min~& Q_{0}\cdot X\\
\textrm{s.t.} ~&Q_{i}\cdot X \leq b_{i}, ~ i=1, \ldots, m, \\
&l_i^2 \leq  X_{ii}=R_{ii} \leq u^2_i, ~ i=1, \ldots, n, \\
&\arg(X_{ij}) \in \mathcal{A}_{ij},~|X_{ij}|=R_{ij},~R_{ij}^2=R_{ii}R_{jj},~ \{i,j\}\in  \mathcal{E},\\
&X\succeq 0.
\end{aligned}
\end{equation}
Notice that constraints $|X_{ij}|=R_{ij}$ and $R_{ij}^2=R_{ii}R_{jj}$ are nonconvex in \eqref{CQP3}.
Hence, we need to relax these nonconvex constraints in order to obtain a convex relaxation. To do this, we define the following set for each $\{i,j\}\in \mathcal{E}$:
\begin{equation}
\mathcal{H}_{ij}=\{(R_{ii},R_{jj},R_{ij}) \mid l_i^2 \leq  R_{ii} \leq u^2_i,l_j^2 \leq  R_{jj} \leq u^2_j, R_{ij}^2=R_{ii}R_{jj}\}.
\end{equation}
We have the following theorem, which is cited from Corollary 5 in \cite{Chen2016}, with some modifications on notations.
\begin{theorem}[\cite{Chen2016}] \label{thm1}
The following two linear inequalities are valid for $\mathcal{H}_{ij}$:
\begin{align}
&(l_i+u_i)(l_j+u_j)R_{ij}\geq (l_j^2+l_j u_j)R_{ii}+(l_i^2 +l_i u_i)R_{jj}+l_i l_j u_i u_j -l_i^2 l_j^2, \label{eq1} \\
&(l_i+u_i)(l_j+u_j)R_{ij}\geq (u_j^2+l_j u_j)R_{ii}+(u_i^2 +l_i u_i)R_{jj}+l_i l_j u_i u_j -u_i^2 u_j^2. \label{eq2}
\end{align}
Moreover, the convex hull of $\mathcal{H}_{ij}$ can be represented as follows:
\begin{align}
\textsf{Conv}(\mathcal{H}_{ij})=\left\{ (R_{ii},R_{jj},R_{ij}) \,\left|\,
\begin{array}{@{}lll} l_i^2 \leq  R_{ii} \leq u^2_i,
l_j^2 \leq  R_{jj} \leq u^2_j\\ R_{ij}^2\leq R_{ii}R_{jj}\\
(R_{ii},R_{jj},R_{ij})~\textrm{satisfies \eqref{eq1} and \eqref{eq2}}\end{array}\right. \right\}.
\end{align}
\end{theorem}

The inequality $R_{ij}^2\leq R_{ii}R_{jj}$ can be transformed to $(2R_{ij}^2+R_{ii}^2+R_{jj}^2)^{1/2}\leq R_{ii}+R_{jj}$.
Hence, the constraint $(R_{ii},R_{jj},R_{ij})\in \textsf{Conv}(\mathcal{H}_{ij})$ can be formulated as a set of linear constraints and a second-order cone constraint.

Next, for any given $R_{ij}\geq 0$, we consider the nonconvex set $$\mathcal{G}_{ij}(R_{ij}):=\{ X_{ij}\mid R_{ij}=|X_{ij}|,~ \arg (X_{ij})\in \mathcal{A}_{ij} \}$$
and its convex hull. When $R_{ij}=0$, we simply define $\mathcal{G}_{ij}(R_{ij})=\{0\}$. When $R_{ij}>0$, by using the similar arguments in \cite{Lu2018}, we have the next two propositions to describe the convex hull of the set $\mathcal{G}_{ij}(R_{ij})$.

\begin{proposition} \label{prop1}
For the case of $\mathcal{A}_{ij}=[\underline {\theta}_{ij},\bar\theta_{ij}]$ where $\bar\theta_{ij}-\underline {\theta}_{ij} < 2\pi,$ we have
\begin{equation}\label{eqgij}
\textsf{Conv} (\mathcal{G}_{ij}(R_{ij}))=\left\{X_{ij}\,|\, a_{ij} \mathrm{Re}\left(X_{ij}\right)+ b_{ij} \mathrm{Im}\left(X_{ij}\right)\geq c_{ij}R_{ij},~|X_{ij}|\leq R_{ij} \right\},
\end{equation}
where
\begin{equation}\label{thm2eq1}
a_{ij}=\cos\left(\frac{\underline {\theta}_{ij}+\bar\theta_{ij}}{2}\right),~b_{ij}=\sin\left(\frac{\underline {\theta}_{ij}+\bar\theta_{ij}}{2}\right), \textrm{and}~c_{ij}=\cos\left(\frac{\bar\theta_{ij}-\underline {\theta}_{ij}}{2}\right).
\end{equation}
\end{proposition}

\begin{proposition} \label{prop2}
For the case of $\mathcal{A}_{ij}=\{\theta_{ij}^1,\theta_{ij}^2,\ldots,\theta_{ij}^M\}$ where $$0\leq \theta_{ij}^1 < \theta_{ij}^2 < \cdots<\theta_{ij}^M<2\pi,$$
we have
\begin{equation}\label{thm2eq3}
\textsf{Conv}(\mathcal{G}_{ij}(R_{ij}))=\left\{X_{ij}\,\left|\, \begin{array}{@{}ll} &a_{ij}^t \mathtt{Re}\left(X_{ij}\right)+ b_{ij}^t \mathtt{Im}\left(X_{ij}\right)\leq c_{ij}^t R_{ij},\\&t=1,2,\ldots,M \end{array}\right.\right\},
\end{equation}
where $\theta_{ij}^{M+1}:=\theta_{ij}^{1}+2\pi$ and
\begin{equation}\label{thm2eq2}
\begin{aligned}
&a_{ij}^t=\cos\left(\frac{\theta_{ij}^t+\theta_{ij}^{t+1}}{2}\right),~b_{ij}^t=\sin\left(\frac{\theta_{ij}^{t}+\theta_{ij}^{t+1}}{2}\right),\\
&c_{ij}^t=\cos\left(\frac{\theta_{ij}^{t+1}-\theta_{ij}^{t}}{2}\right).
\end{aligned}
\end{equation}
\end{proposition}

The proof of the above two propositions are straightforward, and thus are omitted here. We illustrate the two propositions in Figure \ref{fig1}. As illustrated in the left-hand side of Figure \ref{fig1}, for the case $\mathcal{A}_{ij}=[\underline {\theta}_{ij},\bar\theta_{ij}]$ where $\bar\theta_{ij}-\underline {\theta}_{ij} < 2\pi$, the set $\textsf{Conv} (\mathcal{G}_{ij}(R_{ij}))$ is defined by the set enclosed by the straight line passing through points $A$ and $B$, and the arc connecting the two points. Similarly, as illustrated in the right-hand side of Figure \ref{fig1}, $\textsf{Conv} (\mathcal{G}_{ij}(R_{ij}))$ can be represented by $M$ inequalities when $\mathcal{A}_{ij}=\{\theta_{ij}^1,\theta_{ij}^2,\ldots,\theta_{ij}^M\}$.

\begin{figure}
\centering
\includegraphics[height=5.5cm]{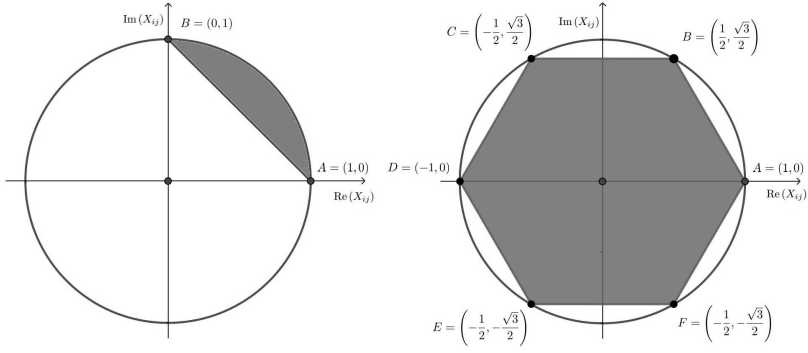}
\centering
\caption{An illustration of the set $\mathcal{G}_{ij}(R_{ij})$ for the case $R_{ij}=1$.}\label{fig1}
\end{figure}

Based on Propositions \ref{prop1} and \ref{prop2}, the constraint $X_{ij}\in \textsf{Conv} (\mathcal{G}_{ij}(R_{ij}))$ can be formulated as a linear inequality constraint and
a second-order cone constraint of the form $(\texttt{Re}^2(X_{ij})+\texttt{Im}^2(X_{ij}))^{1/2}\leq R_{ij}$ for the continuous case, or formulated as $M$
linear inequality constraints for the discrete case.

Based on Theorems \ref{thm1}, Propositions \ref{prop1} and \ref{prop2}, we can relax \eqref{CQP3} to the following convex problem:

\begin{equation}\label{ECSDP}\tag{ECSDP1}
\begin{aligned}
\min~& Q_{0}\cdot X\\
\textrm{s.t.} ~&Q_{i}\cdot X \leq b_{i}, ~i=1, \ldots, m, \\
&l_i^2 \leq  R_{ii}=X_{ii} \leq u^2_i,~ i=1, \ldots, n, \\
& X_{ij} \in \textsf{Conv} (\mathcal{G}_{ij}(R_{ij})), ~\{i,j\}\in  \mathcal{E},\\
& (R_{ii},R_{jj},R_{ij})\in \textsf{Conv}(\mathcal{H}_{ij}),~\{i,j\}\in  \mathcal{E},\\
&X\succeq 0.
\end{aligned}
\end{equation}
In order to study the tightness of \eqref{ECSDP}, we further define the set
\begin{equation}
\mathcal{F}_{ij}=\left\{(X_{ii},X_{jj},R_{ij},X_{ij})\,\left|\,
\begin{array}{@{}lll}  l_i^2 \leq  X_{ii} \leq u^2_i\\
l_j^2 \leq  X_{jj} \leq u^2_j\\
X_{ii}X_{jj}=R^2_{ij}=|X_{ij}|^2\\
\arg (X_{ij})\in \mathcal{A}_{ij} \end{array}\right.
\right\}.
\end{equation}
We have the following theorem.

\begin{theorem}\label{thm3}
Let $(X,R)$ be a feasible solution of \eqref{ECSDP}, then we have
$(X_{ii},X_{jj},R_{ij},X_{ij}) \in \textsf{Conv}(\mathcal{F}_{ij})$ for any $\{i,j\}\in \mathcal{E}$.
\end{theorem}
\begin{proof}
Since $(X_{ii},X_{jj},R_{ij})\in \textsf{Conv}( \mathcal{H}_{ij})$, there exist a sequence of points
\begin{equation}
(X^t_{ii},X^t_{jj},R^t_{ij})\in \mathcal{H}_{ij}, ~t=1,\ldots,r,
\end{equation}
such that
\begin{equation}
(X_{ii},X_{jj},R_{ij})=\sum_{t=1}^r \lambda_t (X^t_{ii},X^t_{jj},R^t_{ij}),
\end{equation}
where $\lambda_1,\ldots,\lambda_r\geq 0$ and $\sum_{i=1}^r \lambda_t=1$.
Similarly, since $X_{ij}\in \textsf{Conv}( \mathcal{G}_{ij}(R_{ij}))$, there exist a sequence of points
\begin{equation}
X^s_{ij}\in  \mathcal{G}_{ij}(R_{ij}), ~s=1,\ldots,k,
\end{equation}
such that
\begin{equation}
X_{ij}=\sum_{s=1}^k \alpha_s X^s_{ij},
\end{equation}
where $\alpha_1,\ldots,\alpha_k\geq 0$ and $\sum_{s=1}^k \alpha_s=1$.
For each $s\in\{1,\ldots,k\}$, denote $ \arg (X^s_{ij})$ by $\theta^s_{ij}$, then we have
$X^s_{ij}=R_{ij}e^{ \texttt{i} \theta_{ij}}$. Now, we construct the following points:
\begin{equation}\label{thm3eq4}
(X^t_{ii},X^t_{jj},R^t_{ij},R^t_{ij} e^{ \texttt{i} \theta^s_{ij}})\in \mathcal{F}_{ij},~s=1,\ldots,k,~t=1,\ldots,r.
\end{equation}
Then, we have that
\begin{equation}\label{thm3eq1}
\sum_{t=1}^r \sum_{s=1}^k \lambda_t \alpha_s (X^t_{ii},X^t_{jj},R^t_{ij})
=\sum_{t=1}^r  \lambda_t (X^t_{ii},X^t_{jj},R^t_{ij})
= (X_{ii},X_{jj},R_{ij}),
\end{equation}
and
\begin{equation}\label{thm3eq2}
\sum_{s=1}^k \sum_{t=1}^r \lambda_t \alpha_s R^t_{ij} e^{ \texttt{i} \theta^s_{ij}}
=\sum_{s=1}^k \alpha_s \left( \sum_{t=1}^r  \lambda_t R^t_{ij}\right) e^{ \texttt{i} \theta^s_{ij}}
=\sum_{s=1}^k \alpha_s R_{ij}e^{ \texttt{i} \theta^s_{ij}}= X_{ij}.
\end{equation}
Following from \eqref{thm3eq1} and \eqref{thm3eq2}, we have
\begin{equation}\label{thm3eq3}
(X_{ii},X_{jj},R_{ij},X_{ij})= \sum_{s=1}^k \sum_{t=1}^r \lambda_t \alpha_s (X^t_{ii},X^t_{jj},R^t_{ij},R^t_{ij} e^{ \texttt{i} \theta^s_{ij}}).
\end{equation}
Since
\begin{equation}
\sum_{s=1}^k \sum_{t=1}^r \lambda_t \alpha_s =\sum_{s=1}^k \left(\sum_{t=1}^r \lambda_t\right) \alpha_s =\sum_{s=1}^k \alpha_s=1,
\end{equation}
Equation \eqref{thm3eq3} implies that $(X_{ii},X_{jj},R_{ij},X_{ij})$ is the convex combination of the $k\cdot r$ points defined in \eqref{thm3eq4}, thus $(X_{ii},X_{jj},R_{ij},X_{ij}) \in \textsf{Conv}(\mathcal{F}_{ij})$.
\qed
\end{proof}

In \eqref{ECSDP}, the variables $(R_{ii},R_{jj},R_{ij})$ are introduced to link the connections between $(X_{ii},X_{jj})$ and $X_{ij}$ by the following constraints:
\begin{equation}\label{eq6}
\begin{aligned}
&X_{ii}=R_{ii},~X_{jj}=R_{jj},\\
&X_{ij} \in \textsf{Conv} (\mathcal{G}_{ij}(R_{ij})),\\
&(R_{ii},R_{jj},R_{ij})\in \textsf{Conv}(\mathcal{H}_{ij}).
\end{aligned}
\end{equation}
In order to analyze whether the constraints in \eqref{eq6} can bridge a strong connection between $(X_{ii},X_{jj})$ and $X_{ij}$, we project the set $\mathcal{F}_{ij}$ onto the space of variables $(X_{ii},X_{jj},X_{ij})$ and define the following set:
\begin{equation}\label{Jij}
\mathcal{J}_{ij}=\left\{(X_{ii},X_{jj},X_{ij})\,\left|\, \begin{array}{@{}lll}
l_i^2 \leq  X_{ii} \leq u^2_i\\
l_j^2 \leq  X_{jj} \leq u^2_j\\
X_{ii}X_{jj}=|X_{ij}|^2\\
\arg (X_{ij})\in \mathcal{A}_{ij}
\end{array}\right.\right\}.
\end{equation}
A direct consequence of Theorem \ref{thm3} is the following theorem.

\begin{theorem}\label{thm4}
Let $(X,R)$ be a feasible solution of \eqref{ECSDP}, then for any $\{i,j\}\in \mathcal{E}$, we have
\begin{equation}\label{thm4eq1}
(X_{ii},X_{jj},X_{ij}) \in \textsf{Conv}(\mathcal{J}_{ij}).
\end{equation}
\end{theorem}
\begin{proof}
Based on equation \eqref{thm3eq3} in the proof of Theorem \ref{thm3}, we can derive the equation on the projected space of variables $(X_{ii},X_{jj},X_{ij})$ as follows:
\begin{equation}
(X_{ii},X_{jj},X_{ij})= \sum_{s=1}^k \sum_{t=1}^r \lambda_t \alpha_s (X^t_{ii},X^t_{jj},R^t_{ij} e^{ \texttt{i} \theta^s_{ij}}),
\end{equation}
where
\begin{equation}
(X^t_{ii},X^t_{jj},R^t_{ij} e^{ \texttt{i} \theta^s_{ij}})\in \mathcal{J}_{ij},~s=1,\ldots,k,~t=1,\ldots,r.
\end{equation}
Thus we have $$(X_{ii},X_{jj},X_{ij}) \in \textsf{Conv}(\mathcal{J}_{ij}).$$
\qed
\end{proof}

Theorem \ref{thm4} shows that by introducing the constraints in \eqref{eq6} into \eqref{ECSDP}, the variables $(X_{ii},X_{jj})$ and $X_{ij}$ are strongly connected,
in the sense that for any feasible solution $(X,R)$ of \eqref{ECSDP}, we have $(X_{ii},X_{jj},X_{ij}) \in \textsf{Conv}(\mathcal{J}_{ij})$ for any $\{i,j\}\in \mathcal{E}$. Thus, the convex hull of $\mathcal{J}_{ij}$ is described by the constraints in \eqref{ECSDP} exactly. Based on the above analysis, we may expect that \eqref{ECSDP} can be tighter than \eqref{CSDP}.
Note that $\mathcal{J}_{ij}$ is equivalent to the set $\mathcal{J}_{ij}^x$ defined in Section 1. The convex hull of $\mathcal{J}_{ij}^x$ is described exactly for general cases of \eqref{CQP}
with different types of $\mathcal{A}_{ij}$.

\section{Comparisons with existing semidefinite relaxations}

Besides of the proposed semidefinite relaxation \eqref{ECSDP}, some previous papers including \cite{Chen2016,Lu2018,Lu2020} have also discussed how to exploit the structure of the
phase angle constraints in \eqref{CQP} to derive tight semidefinite relaxations. In this section, we discuss the relationship between \eqref{ECSDP} and the existing ones in \cite{Chen2016,Lu2018,Lu2020}.

\subsection{Comparison with the semidefinite relaxation in \cite{Chen2016}}

We first discuss the connections between \eqref{ECSDP} and the semidefinite relaxation proposed in \cite{Chen2016}.
Consider the following continuous case of \eqref{CQP}:
\begin{equation}\label{CQP4}
\begin{aligned}
\min~&x^{\dag} Q_{0} x\\
\textrm{s.t.} ~&x^{\dag} Q_{i} x \leq b_{i}, ~ i=1, \ldots, m, \\
&l_i \leq {|}x_i {|} \leq u_i,~ i=1, \ldots, n, \\
&\arg(x_i{x_j}^{\dag}) \in [\underline{\theta}_{ij},\bar{\theta}_{ij}], ~\{i,j\}\in \mathcal{E},
\end{aligned}
\end{equation}
where $-\pi/2<\underline{\theta}_{ij}<\bar{\theta}_{ij}<\pi/2$. Using the notations in \cite{Chen2016}, the lifted matrix $X$ is represented as $X=W+\texttt{i} T$,
and the following set is defined:
\begin{equation}
\mathcal{J}_{C}=\left\{(W_{ii},W_{jj},W_{ij},T_{ij})\,\left|\, \begin{array}{@{}lll}
&l_i^2 \leq  W_{ii} \leq u^2_i\\
&l_j^2 \leq  W_{jj} \leq u^2_j\\
&L_{ij} W_{ij}\leq T_{ij}\leq U_{ij} W_{ij}\\
&W_{ii}W_{jj}= W^2_{ij}+T^2_{ij}
\end{array}\right.\right\},
\end{equation}
where $L_{ij}=\tan \underline{\theta}_{ij}$ and $U_{ij}=\tan \bar{\theta}_{ij}$. It is easy to check that the set $\mathcal{J}_{ij}$ defined in \eqref{Jij} is equivalent to the set $\mathcal{J}_{C}$ under the relationship $$(X_{ii},X_{jj},X_{ij})=(W_{ii},W_{jj},W_{ij}+\texttt{i} T_{ij}).$$ In order to describe the convex hull of $\mathcal{J}_C$, Chen et al. first relax $\mathcal{J}_C$ to the set defined by the following inequalities:
\begin{equation}\label{eq3}
\begin{aligned}
&l_i^2 \leq  W_{ii} \leq u^2_i, \\
&l_j^2 \leq  W_{jj} \leq u^2_j,\\
&L_{ij} W_{ij}\leq T_{ij}\leq U_{ij} W_{ij},\\
&W_{ii}W_{jj}\geq W^2_{ij}+T^2_{ij}.
\end{aligned}
\end{equation}
Then, they derive the following valid inequalities for $\mathcal{J}_{C}$:
\begin{equation}\label{eq4}
\begin{aligned}
& \pi^3_{ij} W_{ij} +\pi^4_{ij} T_{ij} \geq (l_j^2+l_j u_j)W_{ii}+(l_i^2 +l_i u_i)W_{jj}+l_i l_j u_i u_j -l_i^2 l_j^2, \\
& \pi^3_{ij} W_{ij} +\pi^4_{ij} T_{ij} \geq (u_j^2+l_j u_j)W_{ii}+(u_i^2 +l_i u_i)W_{jj}+l_i l_j u_i u_j -u_i^2 u_j^2,
\end{aligned}
\end{equation}
where
\begin{equation}
\begin{aligned}
&\pi_{ij}^3=(l_i+u_i)(l_j+u_j)\frac{ 1- f(L_{ij})f(U_{ij})}{1+ f(L_{ij})f(U_{ij})},\\
&\pi_{ij}^4=(l_i+u_i)(l_j+u_j)\frac{  f(L_{ij})+f(U_{ij})}{1+ f(L_{ij})f(U_{ij})},
\end{aligned}
\end{equation}
and $f(\cdot)$ is defined as
\begin{equation}
\begin{aligned}
f(t):=\left\{\begin{array}{@{}ll}
(\sqrt{1+t^2}-1)/t,~&\textrm{if}~t\neq 0,\\
0,~&\textrm{if}~t=0.
\end{array}\right.
\end{aligned}
\end{equation}
We cite the following results from \cite{Chen2016} (see Propositions 2 and 3 in \cite{Chen2016}).

\begin{theorem}[\cite{Chen2016}]\label{thmchen}
For any $\{i,j\}\in \mathcal{E}$, the system of inequalities in \eqref{eq4} is valid for $\mathcal{J}_{C}$. Moreover, we have
\begin{equation*}
\textsf{Conv} (\mathcal{J}_{C})=\{(W_{ii},W_{jj},W_{ij},T_{ij}) \,|\, (W_{ii},W_{jj},W_{ij},T_{ij})\textrm{~satisfies } \eqref{eq3}~\textrm{and}~ \eqref{eq4}\}.
\end{equation*}
\end{theorem}

Based on Theorem \ref{thmchen}, Chen et al. have proposed the following semidefinite relaxation (named as SDP+complex valid inequalities (3a) and (3b) in \cite{Chen2016}):
\begin{equation}\label{SDPChen}\tag{ECSDP2}
\begin{aligned}
\min~& Q_{0}\cdot X\\
\textrm{s.t.} ~&Q_{i}\cdot X \leq b_{i}, ~ i=1, \ldots, m, \\
&l_i^2 \leq  W_{ii} \leq u^2_i, ~ i=1, \ldots, n, \\
&(W_{ii},W_{jj},W_{ij},T_{ij})\textrm{~satisfies } \eqref{eq4},~\{i,j\}\in \mathcal{E}, \\
& L_{ij} W_{ij}\leq T_{ij}\leq U_{ij} W_{ij},~\{i,j\}\in \mathcal{E},\\
&X=W+\texttt{i} T\succeq 0.
\end{aligned}
\end{equation}
We have the next theorem.

\begin{theorem}\label{thm6}
The two semidefinite relaxations \eqref{ECSDP} and \eqref{SDPChen} are equivalent
under the assumptions that $\mathcal{A}_{ij}=[\underline{\theta}_{ij},\bar{\theta}_{ij}]$ and $-\pi/2<\underline{\theta}_{ij}<\bar{\theta}_{ij}<\pi/2$ for all $\{i,j\}\in \mathcal{E}$.
\end{theorem}
\begin{proof}
We first assume that $(X,R)$ is a feasible solution of \eqref{ECSDP}. Letting $W=\texttt{Re}(X)$ and $T=\texttt{Im}(X)$.
Then based on Theorem \ref{thm4}, for any $\{i,j\}\in \mathcal{E}$, we have $(X_{ii},X_{jj},X_{ij})\in \textsf{Conv}(\mathcal{J}_{ij})$,
which implies $(W_{ii},W_{jj},W_{ij},T_{ij})\in \textsf{Conv}(\mathcal{J}_{C})$. Thus $(W_{ii},W_{jj},W_{ij},T_{ij})$ satisfies
\eqref{eq3} and \eqref{eq4} for $\{i,j\}\in \mathcal{E}$, and $X=W+\texttt{i} T$ is feasible to \eqref{SDPChen}.

Next, we assume that $X=W+\texttt{i} T$ is a feasible solution to \eqref{SDPChen}. Then for any $\{i,j\}\in \mathcal{E}$, based on Theorem \ref{thmchen}, we have
$(W_{ii},W_{jj},W_{ij},T_{ij})\in \textsf{Conv}(\mathcal{J}_{C})$, which implies
$(X_{ii},X_{jj},X_{ij})\in \textsf{Conv}(\mathcal{J}_{ij})$. Then, there exist
a sequence of points
\begin{equation}
(X^t_{ii},X^t_{jj},X^t_{ij})\in \mathcal{J}_{ij},~t=1,\ldots,k,
\end{equation}
such that
\begin{equation}
(X_{ii},X_{jj},X_{ij})=\sum_{t=1}^k \lambda_t (X^t_{ii},X^t_{jj},X^t_{ij}),
\end{equation}
where $\lambda_t\geq 0$ for $t=1,\ldots,k$ and $\sum_{t=1}^k \lambda_t=1$.
We expand the feasible solution $X$ to a feasible solution $(X,R)$ to \eqref{ECSDP} as follows: For each $i\in\{1,\ldots,n\}$, let $R^t_{ii}=|X^t_{ii}|$
for $t=1,\ldots,k$, and $R_{ii}=\sum_{t=1}^k \lambda_t R^t_{ii}$. Similarly, for each $\{i,j\}\in \mathcal{E}$, we assign $R_{ij}=\sum_{t=1}^k \lambda_t R^t_{ij}$, where $R^t_{ij}=|X^t_{ij}|$.
Based on the above constructions, we can check that $(R^t_{ii},R^t_{jj},R^t_{ij})\in \mathcal{H}_{ij}$ for all $t=1,\ldots,k$, so we have
\begin{equation}\label{thm6eq1}
(R_{ii},R_{jj},R_{ij})=\sum_{t=1}^k \lambda_t(R^t_{ii},R^t_{jj},R^t_{ij})\in \textsf{Conv}(\mathcal{H}_{ij}).
\end{equation}
Meanwhile, we have
\begin{equation}\label{thm6eq2}
X_{ij}=\sum_{t=1}^k \lambda_t X^t_{ij} = \sum_{t=1}^k \frac{\lambda_t R^t_{ij}}{R_{ij}} R_{ij} e^{ \texttt{i} \arg (X^t_{ij}) }.
\end{equation}
Since
\begin{equation}\label{thm6eq3}
\sum_{t=1}^k \frac{\lambda_t R^t_{ij}}{R_{ij}}=1,
\end{equation}
Equations \eqref{thm6eq2} and \eqref{thm6eq3} together imply that $X_{ij}$ is a convex combination of the points $R_{ij} e^{ \texttt{i} \arg (X^t_{ij}) } \in \mathcal{G}_{ij}(R_{ij})$,
thus
\begin{equation}\label{thm6eq4}
X_{ij}\in \textsf{Conv} (\mathcal{G}_{ij}(R_{ij})).
\end{equation}
Then, the feasibility of $X$ to \eqref{SDPChen}, together with \eqref{thm6eq1} and \eqref{thm6eq4}, implies that $(X,R)$ is feasible to \eqref{ECSDP}.

Based on the above discussions, we conclude that $X=W+ \texttt{i} T$ is feasible to \eqref{SDPChen} if and only if there exists a matrix $R$ such that $(X,R)$ is feasible to \eqref{ECSDP}, and the two feasible solutions have the same objective value. Thus the two problems are equivalent.
\qed
\end{proof}

Theorem \ref{thm6} shows that when $\mathcal{A}_{ij}=[\underline{\theta}_{ij},\bar{\theta}_{ij}] \subseteq(-\pi/2,\pi/2)$ for all $\{i,j\}\in \mathcal{E}$, the two relaxations \eqref{ECSDP} and \eqref{SDPChen} are equivalent. The main difference between the two relaxations is that the convex hull of $\mathcal{J}_{C}$ in \eqref{SDPChen} does not introduce the matrix $R$ to link variables $(X_{ii},X_{jj})$ and $X_{ij}$, so that \eqref{SDPChen} can be more compact than \eqref{ECSDP}. However, for general cases of $\mathcal{A}_{ij}$, it is not easy to derive the convex hull of $\mathcal{J}_{ij}$ directly. In these cases, it is very meaningful to introduce the matrix $R$ to link the connections between $(X_{ii},X_{jj})$ and $X_{ij}$, from which we may derive the convex hull of $\mathcal{J}_{ij}$ easily. Hence, \eqref{ECSDP} is more general than \eqref{SDPChen}, in the sense that it can be applied to general cases of \eqref{CQP} with different types of $\mathcal{A}_{ij}$. Moreover, even for the case in which $\mathcal{A}_{ij}\subseteq(-\pi/2,\pi/2)$ for all $\{i,j\}\in \mathcal{E}$, introducing the real matrix $R$ is also helpful, since we may further enhance \eqref{ECSDP} by adding a new constraint $R\succeq 0$. We will discuss the effects of adding $R\succeq 0$ to \eqref{ECSDP} in the next section.

\subsection{Comparison with the semidefinite relaxation in \cite{Lu2020}}
Next, we analyze the relationship between \eqref{ECSDP} and the enhanced semidefinite relaxation proposed in \cite{Lu2020}, which includes
the semidefinite relaxation proposed in \cite{Lu2018} as a special case. In \cite{Lu2020}, Lu et al.  have studied the following nonhomogeneous quadratic programming problem:
\begin{equation}\label{CQP5}
\begin{aligned}
\min~&\frac{1}{2} x^{\dag} Q_{0} x + \texttt{Re}(c^\dag x)\\
&l_i \leq {|}x_i {|} \leq u_i, ~ i=1, \ldots, n, \\
&\arg(x_i) \in \mathcal{A}_i\subseteq [0,2\pi], ~i=1, \ldots, n.
\end{aligned}
\end{equation}
The enhanced semidefinite relaxation introduced in \cite{Lu2020} is defined as follows:
\begin{equation}\label{ECSDP3}\tag{ECSDP3}
\begin{aligned}
\min~&\frac{1}{2} Q_{0} \cdot X + \texttt{Re}(c^\dag x)\\
&l_i \leq r_i \leq u_i, ~ i=1, \ldots, n, \\
& x_i \in  \textsf{Conv}(\mathcal{G}_i(r_i)),  ~i=1, \ldots, n, \\
&X_{ii}\geq r_i^2,~ X_{ii}-(l_i+u_i)r_i+l_i u_i\leq 0, ~i=1, \ldots, n,\\
&X\succeq xx^\dag,
\end{aligned}
\end{equation}
where $\mathcal{G}_i(r_i):=\{x_i\,|\, r_i=|x_i|, \arg(x_i) \in \mathcal{A}_i\}$ for $r_i>0$ and $\mathcal{G}_i(0)=\{0\}$. To show the connections between \eqref{ECSDP3} and \eqref{ECSDP}, we first transform \eqref{CQP5} to a homogeneous problem by appending $x_{n+1}=1$ to the vector $x\in \mathbb{C}^n$ to construct an $(n+1)$-dimensional vector $y$, and derive
the following problem:
\begin{equation}\label{CQP6}
\begin{aligned}
\min~&\frac{1}{2} y^{\dag} \tilde{Q}_{0} y\\
&l_i \leq {|}y_i {|} \leq u_i,  ~i=1, \ldots, n, \\
&\arg(y_i y_{n+1}^\dag) \in \mathcal{A}_i,~ i=1, \ldots, n,\\
&y_{n+1}=1,
\end{aligned}
\end{equation}
where
\begin{align*}\tilde{Q}_{0}=\begin{bmatrix}
Q_0 ~&c\\
c^{\dag} ~&0\\
\end{bmatrix}\in \mathbb{C}^{(n+1)\times(n+1)}.
\end{align*}
Then, \eqref{ECSDP} for \eqref{CQP6} is formulated as follows:
\begin{equation}\label{ECSDP4}\tag{ECSDP4}
\begin{aligned}
\min~&\frac{1}{2} \tilde{Q}_{0} \cdot Y\\
&l^2_i \leq Y_{ii} \leq u^2_i, ~ i=1, \ldots, n, \\
&Y_{i,n+1}\in \textsf{Conv}(\mathcal{G}_{i,n+1}(R_{i,n+1})),\\
& Y_{ii}\geq R_{i,n+1}^2,~ Y_{i,i}-(l_i+u_i)R_{i,n+1}+l_i u_i\leq 0, ~i=1, \ldots, n,\\
&Y_{n+1,n+1}=1,~Y\succeq 0,
\end{aligned}
\end{equation}
where the inequality $Y_{i,i}-(l_i+u_i)R_{i,n+1}+l_i u_i\leq 0$ is derived from \eqref{eq1} and \eqref{eq2}
with $l_{n+1}=u_{n+1}=1$. Then we have the next theorem.

\begin{theorem}\label{thm7}
Problem \eqref{ECSDP4} is equivalent to \eqref{ECSDP3}.
\end{theorem}
\begin{proof}
Consider the following correspondence relationship between the solutions of \eqref{ECSDP3} and \eqref{ECSDP4}:
\begin{equation}
\begin{aligned}
&Y=\begin{bmatrix}
X ~&x\\
x^{\dag} ~&1\\
\end{bmatrix},\\
&r_i=R_{i,n+1},~i=1,\ldots,n.
\end{aligned}
\end{equation}
It is easy to check that $(Y,R)$ is feasible to \eqref{ECSDP4} if and only if the corresponding solution $(X,x,r)$
is feasible to \eqref{ECSDP3}, and the two solutions have the same objective value. Thus, the two problems are equivalent.
\qed
\end{proof}

Theorem \ref{thm7} shows the equivalence between \eqref{ECSDP3} and \eqref{ECSDP4}. Moreover, the following theorem shows that \eqref{ECSDP3} is also equivalent to \eqref{CSDP} on certain cases of \eqref{CQP5}.
\begin{theorem}\label{thm8}
For problem \eqref{CQP5}, under the assumptions that $c=0$ and $0\in \textsf{Conv} (\mathcal{G}_i(r_i))$ for any $r_i\geq 0$, \eqref{ECSDP3} is equivalent to \eqref{CSDP}.
\end{theorem}
\begin{proof}
Since \eqref{ECSDP3} is equivalent to \eqref{ECSDP4}, it is at least as tight as \eqref{CSDP}. On the other hand, under the given assumptions, let $X^*$ be the optimal solution of \eqref{CSDP},
we can always extend it to a solution $(X^*,x^*,r^*)$ of \eqref{ECSDP3} by assigning $x_i^*=0$ and $r^*_i=(X^*_{ii})^{1/2}$ for $i=1,\ldots,n$, and the two solutions
have the same objective value when $c=0$. Thus \eqref{ECSDP3} is equivalent to \eqref{CSDP}.
\qed
\end{proof}

Theorem \ref{thm8} shows that for certain types of homogeneous quadratic programming problems, both the semidefinite relaxations \eqref{ECSDP3} and \eqref{ECSDP4} can not be tighter than the conventional semidefinite relaxation \eqref{CSDP}. In many applications in signal processing \cite{Demir2014,Demir2015,Maio2009}, the problem is just the case that satisfies the assumptions in Theorem \ref{thm8}.
However, \eqref{ECSDP4} is only a special case of \eqref{ECSDP} for which the set $\mathcal{E}$ is predefined as $$\mathcal{E}=\{\{1,n+1\},\ldots,\{n,n+1\}\}.$$
In fact, we may select some other sets $\mathcal{E}$ to define \eqref{ECSDP}, by which it is possible to derive a semidefinite relaxation that is tighter than \eqref{ECSDP4}. For example, for the case that $\mathcal{A}_{i}=\mathcal{A}^M$ for all $i=1,\ldots,n$, where $M\geq 3$, the constraints $\arg(x_i)\in\mathcal{A}^M$ for $i=1,\ldots,n$ imply that $\arg(x_i x_j^\dag)\in \mathcal{A}^M$ for any $1\leq i<j\leq n$. Hence, we may select $\mathcal{E}$ as any subset of $\{\{i,j\}\,|\,1\leq i <j\leq n+1\}$ to define \eqref{ECSDP}. As will be illustrated in the numerical experiments in Section 5, \eqref{ECSDP} can still be much tighter than \eqref{CSDP} on the homogeneous cases even when the assumptions in Theorem \ref{thm8} are satisfied, if the set $\mathcal{E}$ is well selected.

\section{Further discussions on the enhanced semidefinite relaxation}

\eqref{ECSDP} can be further enhanced by introducing a new positive semidefinite constraint $R\succeq 0$. In this section, we define the following enhanced semidefinite relaxation to analyze the effects of introducing $R\succeq 0$ into \eqref{ECSDP}.

\begin{equation}\label{ECSDP2}\tag{ECSDP}
\begin{aligned}
\min~& Q_{0}\cdot X\\
\textrm{s.t.} ~&Q_{i}\cdot X \leq b_{i}, ~i=1, \ldots, m, \\
&l_i^2 \leq  R_{ii}=X_{ii} \leq u^2_i,~ i=1, \ldots, n, \\
& X_{ij} \in \textsf{Conv}( \mathcal{G}_{ij}(R_{ij})), ~\{i,j\}\in \mathcal{E},\\
&(R_{ii},R_{jj},R_{ij})~\textrm{satisfies \eqref{eq1} and \eqref{eq2}},~\{i,j\}\in \mathcal{E}, \\
&X\succeq 0,~R\succeq 0.
\end{aligned}
\end{equation}
Note that since the constraint $R\succeq 0$ implies the constraints $R_{ij}^2\leq R_{ii}R_{jj}$ for all $\{i,j\}\in \mathcal{E}$, the latter constraints can be ignored after adding $R\succeq 0$
into \eqref{ECSDP} to derive the formulation of \eqref{ECSDP2}. However, in the converse direction, since the constraints $R_{ij}^2\leq R_{ii}R_{jj}$ for all $\{i,j\}\in \mathcal{E}$ do not imply $R\succeq 0$, \eqref{ECSDP2} can be tighter than \eqref{ECSDP} for some instances of \eqref{CQP}. The following example demonstrates this claim.

Consider the following $3$-dimensional instance of \eqref{CQP}:
\begin{equation}
\begin{aligned}
\min~& x^\dag Q_{0} x\\
\textrm{s.t.} ~&1 \leq  |x_i| \leq 4,~ i=1,2,3, \\
& \arg(x_i^{\dag}x_j)\in [-\pi/6,\pi/6], ~\{i,j\}\in \mathcal{E},
\end{aligned}
\end{equation}
where
\begin{equation}
\begin{aligned}
Q_0=\begin{bmatrix}
-2 ~&-4 ~&0\\
-4 ~&2 ~&-2\\
0 ~&-2 ~&6\\
\end{bmatrix}+ \texttt{i} \cdot
\begin{bmatrix}
0 ~&-8 ~&1\\
8 ~&0 ~&-10\\
-1 ~&10 ~&0\\
\end{bmatrix}
\end{aligned}
\end{equation}
and $\mathcal{E}=\{\{i,j\}\,|\,1\leq i<j\leq 3\}$. A direct computation shows that the optimal values of \eqref{ECSDP} and \eqref{SDPChen} are both equal to $-248.39$, whereas the optimal value of \eqref{ECSDP2} is $-248.15$. The absolute difference between the optimal values of \eqref{ECSDP2} and \eqref{ECSDP} is $0.24$.
Hence, the proposed relaxation \eqref{ECSDP2} is strictly tighter than the semidefinite relaxations \eqref{ECSDP} and \eqref{SDPChen}.

When deriving the formulation \eqref{ECSDP2} from \eqref{ECSDP}, the second-order cone constraints of the form $R_{ii}R_{jj}\geq R^2_{ij}$ have been dropped. In fact, under certain conditions, more constraints can be dropped to simplify \eqref{ECSDP2}. We consider two such cases based on the next theorem.

\begin{theorem}\label{thm5}
For a pair of $\{i,j\}\in \mathcal{E}$, if $$0 \in \textsf{Conv} (\mathcal{G}_{ij}(R_{ij}))$$ for all $R_{ij}\geq 0$, then
\begin{equation} \label{thm5eq3}
\textsf{Conv} (\mathcal{G}_{ij}(R_{ij}))\subseteq \textsf{Conv} (\mathcal{G}_{ij}(R'_{ij}))
\end{equation}
holds for any $0\leq R_{ij} < R'_{ij}$.
\end{theorem}
\begin{proof}
For the case of $R_{ij}=0$, we have $\textsf{Conv} (\mathcal{G}_{ij}(R_{ij}))=\{0\}$. Then \eqref{thm5eq3} holds under the condition $0 \in \textsf{Conv} (\mathcal{G}_{ij}(R'_{ij}))$.
Now we consider the case of $R_{ij}> 0$. For any $X_{ij}\in \textsf{Conv} (\mathcal{G}_{ij}(R_{ij}))$, there exist a sequence of points
\begin{equation}
X^t_{ij}\in \mathcal{G}_{ij}(R_{ij}),~\lambda_t\geq 0,~t=1,\ldots,k,
\end{equation}
such that
\begin{equation}
X_{ij}=\sum_{t=1}^k \lambda_t X^t_{ij}~\textrm{and}~\sum_{i=1}^t \lambda_t=1.
\end{equation}
For any $R'_{ij}>R_{ij}$, by the definition of $\mathcal{G}_{ij}(R'_{ij})$, it is easy to check that
\begin{equation}\label{thm5eq2}
R'_{ij}X^t_{ij}/R_{ij}\in \mathcal{G}_{ij}(R'_{ij}).
\end{equation}
Then,
\begin{equation}\label{thm5eq1}
R'_{ij}X_{ij}/R_{ij}=\sum_{t=1}^k \lambda_t R'_{ij}X^t_{ij}/R_{ij}\in \textsf{Conv}(\mathcal{G}_{ij}(R'_{ij})).
\end{equation}
Equation \eqref{thm5eq1} and the condition $0\in \textsf{Conv}(\mathcal{G}_{ij}(R'_{ij}))$ together imply that
$X_{ij}\in \textsf{Conv}(\mathcal{G}_{ij}(R'_{ij}))$, since $X_{ij}$ is on the line segment which connects $0$ and $R'_{ij}X_{ij}/R_{ij}$ when $R'_{ij}>R_{ij}$.
This completes the proof.
\qed
\end{proof}

\begin{remark}
The condition $0 \in \textsf{Conv} (\mathcal{G}_{ij}(R_{ij}))$ holds for any $R_{ij}\geq 0$ if and only if there exists a $\tilde{R}_{ij}>0$ such that $0 \in \textsf{Conv} (\mathcal{G}_{ij}(\tilde{R}_{ij}))$.
\end{remark}

Based on Theorem \ref{thm5}, we can show that inequalities \eqref{eq1} and \eqref{eq2} are redundant in \eqref{ECSDP2} when $0 \in \textsf{Conv} (\mathcal{G}_{ij}(R_{ij}))$ for any $R_{ij}\geq 0$, in the sense that the projection of the feasible domain of \eqref{ECSDP2} onto the space of $X$ is not affected by these inequalities. More specifically, the inequalities in \eqref{eq1} and \eqref{eq2} constrain $R_{ij}$ from below, and the inequality $R_{ii}R_{jj}\geq R^2_{ij}$ constrains $R_{ij}$ from above, which provides an upper bound $\bar{R}_{ij}:=(R_{ii}R_{jj})^{1/2}$ of $R_{ij}$. Under the conditions in Theorem \ref{thm5}, the equation $$\bigcup_{R_{ij}\in[\underline{R}_{ij},\bar{R}_{ij}] }\textsf{Conv}(\mathcal{G}_{ij}(R_{ij}))= \textsf{Conv}(\mathcal{G}_{ij}(\bar{R}_{ij}))$$
holds for any $0\leq \underline{R}_{ij} \leq \bar{R}_{ij}$. Thus, dropping the inequalities \eqref{eq1} and \eqref{eq2}, only the lower bound of $R_{ij}$ is affected, but the upper bound $\bar{R}_{ij}$ does not change, so that the range of $X_{ij}$ is unchanged.
Based on the above discussions, we simplify \eqref{ECSDP2} for the following two cases.

\textbf{Case 1:} $\mathcal{A}_{ij}=\mathcal{A}^M$ for all $\{i,j\}\in \mathcal{E}$ where $M\geq 3$. In this case, we can see that $0 \in \textsf{Conv} (\mathcal{G}_{ij}(R_{ij}))$ holds for any $R_{ij}\geq 0$. Hence, based on Theorem \ref{thm5}, we can drop the constraints defined by inequalities \eqref{eq1} and \eqref{eq2} to simplify \eqref{ECSDP2}.

\textbf{Case 2:} $\mathcal{A}_{ij}=[\underline{\theta}_{ij},\bar{\theta}_{ij}]$ with $\pi\leq \bar{\theta}_{ij}-\underline{\theta}_{ij}<2\pi$ for all $\{i,j\}\in \mathcal{E}$.
In this case, we also have $0\in \textsf{Conv} (\mathcal{G}_{ij}(R_{ij}))$. Then, based on Theorem \ref{thm5}, constraints \eqref{eq1} and \eqref{eq2} can be dropped.
In addition, based on Theorem \ref{thm5} again, the constraints $|X_{ij}|\leq R_{ij}$ for $\{i,j\}\in \mathcal{E}$ can be dropped. In detail, let $(\hat{X},\hat{R})$ be the optimal
solution to \eqref{ECSDP2} with constraints \eqref{eq1}, \eqref{eq2} and constraints $|X_{ij}|\leq R_{ij}$ for $\{i,j\}\in \mathcal{E}$ being dropped. Since
the inequalities $|\hat{X}_{ij}|\leq (\hat{X}_{ii}\hat{X}_{jj})^{1/2}$ holds automatically under the constraint $\hat{X}\succeq 0$, we can construct another solution
$(\hat{X},R')$ by assigning the entries of $R'$ to $R'_{ij}=(X_{ii}X_{jj})^{1/2}$. Then, it is easy to check $(\hat{X},R')$ is feasible to \eqref{ECSDP2} in its original version.
Hence, the constraints $|X_{ij}|\leq R_{ij}$ for $\{i,j\}\in \mathcal{E}$ can be dropped without affecting the optimal value of \eqref{ECSDP2}.

In our numerical experiment, in order to reduce the computational complexity, we always simplify \eqref{ECSDP2} according to the types of the phase angle constraints for a variety of \eqref{CQP} arising in different application backgrounds.

\section{Numerical results}
We carry out numerical experiments to investigate the tightness of the proposed semidefinite relaxation \eqref{ECSDP2}. Our test instances are randomly generated to simulate some practical applications in signal processing. The experiments are carried out on a personal computer with Intel Core(TM) i7-9700 CPU (3.00 GHz) and 16 GB RAM. We use Mosek (Ver 9.2) \cite{mosek} to solve all semidefinite relaxations. All algorithms are implemented in Matlab R2017a. For all the three problems discussed in the following subsections, the set $\mathcal{E}$ is always set to $\mathcal{E}=\{\{i,j\}\,|\,1\leq i<j\leq n\}$.

\subsection{Phase quantized waveform design}

We consider the phase quantized waveform design problem with constraints on peak to average ratio (ref. \cite{Maio2011}). The problem can be formulated as follows:
\begin{equation}\label{radar}
\begin{aligned}
\max~&x^{\dag} Q x\\
\textrm{s.t.} ~&x^{\dag} x =n,\\
& {|}x_i {|}^2 \leq \gamma, ~ i=1, \ldots, n, \\
&\arg(x_i) \in \mathcal{A}^M,~ i=1, \ldots, n,
\end{aligned}
\end{equation}
where the meaning of parameters $Q,~n,~\gamma,~M$ and the decision variables $x_1,\ldots,x_n$ are described in \cite{Maio2011}.
Based on the discussions in Section 4, relaxation \eqref{ECSDP2} for \eqref{radar} can be simplified as follows:
\begin{equation}\label{ECSDPradar}
\begin{aligned}
\max~& Q\cdot X\\
\textrm{s.t.} ~&\texttt{Trace}(X)=n,\\
&X_{ii}=R_{ii} \leq \gamma, ~i=1, \ldots, n, \\
& X_{ij} \in \textsf{Conv} (\mathcal{G}_{ij}(R_{ij})), ~\{i,j\}\in \mathcal{E},\\
&X\succeq 0,~R\succeq 0,
\end{aligned}
\end{equation}
where the constraint $X_{ij} \in \textsf{Conv} (\mathcal{G}_{ij}(R_{ij}))$ is described in \eqref{thm2eq3}. Besides, the classical semidefinite relaxation \eqref{CSDP} of \eqref{radar} can be obtained by dropping the constraint $X_{ij} \in \textsf{Conv} (\mathcal{G}_{ij}(R_{ij}))$ for $\{i,j\}\in \mathcal{E}$ in \eqref{ECSDPradar}.
In \cite{Maio2011}, Maio et al. have applied \eqref{CSDP} to design an approximation algorithm.

In our experiments, we compare \eqref{ECSDP2} with \eqref{CSDP} from two aspects: First, we investigate whether \eqref{ECSDP2} can be much tighter than \eqref{CSDP}. Second, we investigate whether the approximation algorithm proposed in \cite{Maio2011} can be improved to find better sub-optimal solutions by using \eqref{ECSDP2} rather than \eqref{CSDP}.

The numerical experiments are carried out as follows: we use ten randomly generated test instances, in each of which the matrix $Q\in \mathbb{C}^{n\times n}$ is generated using the procedures in \cite{Soltanalian}: $Q=\sum_{k=1}^n u_k u_k^\dag$, where $u_k$, $k=1,\ldots,n$, is a random vector in $\mathbb{C}^{n}$ whose real-part and imaginary-part elements are independently sampled from the standard Gaussian distribution. The parameter $\gamma$ is set to $1.2$. We consider the cases of $n=20$ and  $M\in\{3,6\}$. For each test instance, we solve the relaxations \eqref{CSDP} and \eqref{ECSDP2} to estimate upper bounds (UB in short) for \eqref{radar}. Based on the optimal solutions of the two relaxations, we run the approximation algorithm proposed in \cite{Maio2011} to obtain sub-optimal solutions, whose objective values provide lower bounds (LB in short) for \eqref{radar}. The lower and uppder bounds obtained from the two relaxations are listed in Table \ref{table1}, along with the computational time (in seconds) for solving different semidefinite relaxations. Finally, we define the ``Gap Closed" as\footnote{In the literature, the term ``gap" means the difference between the optimal value of an optimization problem and its lower/upper bound. Here we borrow the term ``gap" for convenience, but the meaning is different.}
\begin{equation}
\textrm{Gap Closed}=1-\frac{\textrm{UB}_{\textrm{ECSDP}}-\textrm{LB}_{\textrm{ECSDP}}}{\textrm{UB}_{\textrm{CSDP}}-\textrm{LB}_{\textrm{CSDP}}},
\end{equation}
where $\textrm{LB}_{(\bullet)}$ and $\textrm{UB}_{(\bullet)}$ denote the lower and upper bounds obtained from the corresponding relaxation method, respectively.
The results of Gap Closed are listed in the last column of Table \ref{table1}.

\begin{table}[ht]
\begin{center}
\begin{minipage}{\textwidth}
\caption{Computational results on instances of Problem \eqref{radar}.} \label{table1}
\begin{tabular*}{\textwidth}{@{\extracolsep{\fill}}ccrrrrrrr@{\extracolsep{\fill}}}
\toprule%
\multicolumn{2}{@{}c@{}}{Instance} & \multicolumn{3}{@{}c@{}}{CSDP} & \multicolumn{3}{@{}c@{}}{ECSDP}  &Gap \\ \cmidrule{3-5}\cmidrule{6-8}
ID &$(n,M)$ &\rm{UB} &\rm{LB}  &\rm{Time}  &\rm{UB} &\rm{LB}  &\rm{Time} &Closed \\
\midrule
1 &(20,3) &1094.58 &1044.10 &0.16 &1071.85 &1050.14 &0.31 &57\% \\
2 &(20,3) &1125.75 &1063.21 &0.17 &1078.46 &1071.90 &0.32 &90\% \\
3 &(20,3) &1137.79 &1079.27 &0.16 &1103.90 &1081.97 &0.32 &63\% \\
4 &(20,3) &1073.05 &1015.17 &0.17 &1049.16 &1020.37 &0.32 &50\% \\
5 &(20,3) &1134.78 &1040.75 &0.16 &1084.89 &1057.29 &0.31 &71\% \\
6 &(20,6) &1130.19 &1025.88 &0.16 &1124.07 &1031.64 &0.59 &11\% \\
7 &(20,6) &1115.55 &1033.44 &0.17 &1112.86 &1040.41 &0.63 &12\% \\
8 &(20,6) &1106.84 &1020.79 &0.17 &1103.74 &1040.53 &0.65 &27\% \\
9 &(20,6) &1126.50 &1024.32 &0.17 &1113.95 &1028.55 &0.62 &16\% \\
10 &(20,6) &1122.23 &1045.24 &0.17 &1110.84 &1048.37 &0.60 &19\% \\
\bottomrule
\end{tabular*}
\end{minipage}
\end{center}
\end{table}

From the results listed in Table \ref{table1}, we can see that the upper bounds provided by \eqref{ECSDP2} are consistently tighter than the bounds of \eqref{CSDP} on all test instances.
For the case of $M=3$, the improvement on closing gaps is significant. In fact, the set $\mathcal{A}^M$ can be regarded as a discrete
approximation of the set $[0,2\pi]$. Hence, for the case that $M$ is small, the constraint $X_{ij} \in \textsf{Conv} (\mathcal{G}_{ij}(R_{ij}))$ can be more
effective for reducing the gap. As $M$ becomes larger, the difference between the optimal values of \eqref{ECSDP2} and \eqref{CSDP} tends to zero. That is
why the tightness of the upper bounds can be improved more significantly on the case of $M=3$ than on the case of $M=6$.

On the other hand, for all the ten test instances, the lower bounds obtained by using \eqref{ECSDP2} are larger than the lower bounds obtained by using \eqref{CSDP}.
These results imply that by using \eqref{ECSDP2} for the approximation algorithm, better sub-optimal solutions can be found.

The improvement on the upper and lower bounds are not obtained for free. As we may observe that the computational time for solving \eqref{ECSDP2} is longer than the one for \eqref{CSDP}. This is reasonable since the number of variables and the number of constraints in \eqref{ECSDP2} are both larger than those of \eqref{CSDP}.

In conclusion, using \eqref{ECSDP2} to replace \eqref{CSDP} for problem \eqref{radar}, we may obtain better upper bounds and better sub-optimal solutions, with the cost of higher computation complexity. Moreover, we would like to mention that problem \eqref{radar} is a homogeneous case of \eqref{CQP} that satisfies the assumptions in Theorem \ref{thm8}, hence the semidefinite relaxation proposed in \cite{Lu2020} can not be tighter than \eqref{CSDP}.

\subsection{Discrete transmit beamforming}
In this subsection, we consider the discrete transmit beamforming problem, which can be formulated as follows (ref. \cite{Demir2015}):
\begin{equation}\label{beamforming}
\begin{aligned}
\max_{t,x}~&t\\
\textrm{s.t.} ~&x^{\dag} Q_{i} x \geq t \gamma_k \sigma_k^2, ~i=1,\ldots,k,\\
&x^{\dag} x\leq  P_{\textrm{tot}},\\
& {|}x_i {|} \in \{\Delta,2\Delta,\ldots, 2^m \Delta \}, ~ i=1, \ldots, n, \\
&\arg(x_i) \in \mathcal{A}^M, i=1, \ldots, n,
\end{aligned}
\end{equation}
where $Q_i=h_i h_i^\dag$ with $h_i\in \mathbb{C}^n$ being the channel vector for each $i=1,\ldots,k$, $n$ is the number of  transmit antennas, $k$ is the number of receivers, $m$ is the number of bits to represent the discrete amplitude values, $P_{\textrm{tot}}$ denotes the maximum total power, $\Delta=\sqrt{P_{\textrm{max}}}/2^m$ where $P_{\textrm{max}}$ is the maximum per-antenna power, and $M=2^b$ where $b$ is the number of bits to represent the discrete amplitude values.

If we relax the discrete constraint ${|}x_i {|} \in \{\Delta,2\Delta,\ldots, 2^m \Delta \}$ to $\Delta\leq {|}x_i {|}  \leq 2^m \Delta$, then we have the  relaxation \eqref{CSDP} for \eqref{beamforming} as follows:
\begin{equation}
\begin{aligned}
\max~&t\\
\textrm{s.t.} ~& Q_{i} \cdot X \geq t \gamma_k \sigma_k^2, ~i=1,\ldots,k,\\
&\texttt{Trace}(X)\leq  P_{\textrm{tot}},\\
& \Delta^2\leq X_{ii}\leq P_{\textrm{max}},  i=1, \ldots, n, \\
& X\succeq 0.
\end{aligned}
\end{equation}
On the other hand, relaxation \eqref{ECSDP2} for \eqref{beamforming} is formulated as follows:
\begin{equation}
\begin{aligned}
\max~&t\\
\textrm{s.t.} ~& Q_{i} \cdot X \geq t \gamma_k \sigma_k^2, ~i=1,\ldots,k,\\
&\texttt{Trace}(X)\leq  P_{\textrm{tot}},\\
& \Delta^2\leq R_{ii}=X_{ii}\leq P_{\textrm{max}},  i=1, \ldots, n, \\
& X_{ij} \in \textsf{Conv} (\mathcal{G}_{ij}(R_{ij})), ~\{i,j\}\in \mathcal{E},\\
& X\succeq 0,~R\succeq 0.
\end{aligned}
\end{equation}

We compare the performance of the two relaxations on instances that are randomly generated according to the procedure in \cite{Demir2015}: For each $i=1,\ldots,k$, the vector $h_i\in \mathbb{C}^n$ follows the standard $n$-dimensional complex Gaussian distribution, $Q_i=h_i h_i^\dag$, $\gamma_k$ is uniformly selected from $\{1,2,3,4\}$, and $\sigma_k^2=1.0$. {Also based on the parameter settings in \cite{Demir2015}, we set $n=4$, $P_{\textrm{max}}=20$, $P_{\textrm{tot}}=40$, and consider the case of $k\in\{4,8,12\}$, $m\in\{3,4\}$, and $b\in\{3,4\}$.

Based on the procedures described above, for each setting of parameters $(n,k,m,b)$, we generate ten test instances. Each instance is relaxed to \eqref{ECSDP2} and \eqref{CSDP}, respectively.
Moreover, the optimal value of each instance is computed by CPLEX, using the integer programming formulation proposed in \cite{Demir2015}. With a known optimal value for each test instance, the ``Gap Closed'' is defined as
\begin{equation}
\textrm{Gap Closed}=1-\frac{\textrm{UB}_{\textrm{ECSDP}}-V^\ast}{\textrm{UB}_{\textrm{CSDP}}-V^\ast},
\end{equation}
where $V^\ast$ denotes the optimal value returned by CPLEX. The numerical results are listed in Table \ref{table2}. For each row, the listed results summarize the average performance on the ten test instances.

\begin{table}[h]
\begin{center}
\begin{minipage}{\textwidth}
\centering
\caption{Computational results on discrete transmit beamforming problems.}\label{table2}
\begin{tabular}{ccrrrrr}
\toprule%
\multirow{2}{*}{$(n,k,m,b)$} &\multicolumn{2}{c}{ECSDP} &\multicolumn{2}{c}{CSDP} &\multirow{2}{*}{Optimal Value} &\multirow{2}{*}{Gap Closed}\\
\cmidrule(r){2-3} \cmidrule(r){4-5}
&Upper Bound  &Time &Upper Bound &Time \\
\midrule
(4,4,3,3) &   38.7921   &  0.010  &  39.8886  &  0.005  &  36.3148 &  30.16\% \\
(4,8,3,3) &   36.7337   &  0.010  &  37.2159  &  0.005  &  32.1290 &   9.72\% \\
(4,12,3,3) &  25.9584   &  0.011  &  26.2780  &  0.005  &  21.1069 &  10.34\% \\
(4,4,4,4) &   39.3492   &  0.017  &  39.8886  &  0.005  &  38.3498 &  31.58\% \\
(4,8,4,4) &   36.9279   &  0.017  &  37.2159  &  0.005  &  34.1607 &  10.67\% \\
(4,12,4,4) &   26.0792   &  0.017  &  26.2780  &  0.005  &  22.7450 &  10.61\% \\
\bottomrule
\end{tabular}
\footnotetext{Note: The results in each row are averaged over the ten test instances.}
\end{minipage}
\end{center}
\end{table}

From the results in Table \ref{table2}, we can see that relaxation \eqref{ECSDP2} is consistently tighter than \eqref{CSDP} for all settings of parameters. The proposed valid inequalities exploited from the convex hull of $\mathcal{G}_{ij}(R_{ij})$ and $\mathcal{H}_{ij}$ are effective for improving the tightness of the semidefinite relaxation, which reduce 9.7\%--30\% of the relaxation gap introduced by \eqref{CSDP} on average.

}

\subsection{Continuous phase angle constraints}

We consider \eqref{CQP} with continuous phase angle constraints. As discussed in Section 4, when $\mathcal{A}_{ij}=[\underline{\theta}_{ij},\bar{\theta}_{ij}]$ with $-\pi/2<\underline{\theta}_{ij}<\bar{\theta}_{ij}<\pi/2$ for all $\{i,j\}\in \mathcal{E}$, \eqref{ECSDP} is equivalent to \eqref{SDPChen}. However, it has been shown that \eqref{ECSDP2} can be tighter than \eqref{ECSDP}. In this subsection, we compare the tightness of these different relaxations numerically. For this purpose, we generate test instances of the following form:

\begin{equation}
\begin{aligned}
\min~&x^{\dag} Q_{0} x\\
&l_i \leq {|}x_i {|} \leq u_i,~  i=1, \ldots, n, \\
&\arg(x_i{x_j}^{\dag}) \in [\underline{\theta}_{ij},\bar{\theta}_{ij}], \{i,j\}\in \mathcal{E},
\end{aligned}
\end{equation}
where $l_i=1.0$, $u_i=4.0$ for all $i=1,\ldots,n$, and $Q_0\in \mathbb{C}^n$ is a randomly generated Hermitian matrix, with each diagonal entry following the standard Gaussian distribution, and with the real-part and the imaginary-part of each non-diagonal entry also following the standard Gaussian distribution.

We first generate ten 20-dimensional test instances with $\mathcal{A}_{ij}=[-\pi/6,\pi/6]$ for all $\{i,j\}\in \mathcal{E}$, and compare the three relaxations \eqref{CSDP}, \eqref{SDPChen} and \eqref{ECSDP2}. Since \eqref{ECSDP} is equivalent to \eqref{SDPChen} but is not as compact as \eqref{SDPChen}, it is not selected for the current comparison. The computational results for the three selected semidefinite relaxations are listed in Table \ref{table4}.

\begin{table}[h]
\begin{center}
\begin{minipage}{\textwidth}
\caption{Computational results on instances with continuous phase angle set.}\label{table4}
\begin{tabular*}{\textwidth}{@{\extracolsep{\fill}}lrrrrrr@{\extracolsep{\fill}}}
\toprule%
Instance & \multicolumn{3}{@{}c@{}}{Lower Bound} & \multicolumn{3}{@{}c@{}}{Time} \\ \cmidrule{2-4}\cmidrule{5-7}
ID  &\rm{CSDP} &\rm{ECSDP2} &\rm{ECSDP}   &\rm{CSDP} &\rm{ECSDP2} &\rm{ECSDP} \\
\midrule
1 &-2965.66 &-1120.93 &-1116.66 &0.02 &0.20 &0.32 \\
2 &-2895.56 &-593.77 &-592.48 &0.01 &0.16 &0.26 \\
3 &-2567.41 &-1168.01 &-1167.79 &0.01 &0.18 &0.38 \\
4 &-3062.55 &-622.25 &-618.21 &0.01 &0.17 &0.24 \\
5 &-3469.90 &-896.06 &-893.29 &0.01 &0.16 &0.27 \\
6 &-3069.32 &-1026.03 &-1025.35 &0.01 &0.19 &0.27 \\
7 &-2779.32 &-975.31 &-969.00 &0.01 &0.17 &0.25 \\
8 &-2658.00 &-542.72 &-539.58 &0.01 &0.15 &0.24 \\
9 &-3381.43 &-1225.68 &-1221.91 &0.01 &0.18 &0.30 \\
10 &-3087.94 &-1096.40 &-1092.99 &0.01 &0.17 &0.25 \\
\bottomrule
\end{tabular*}
\end{minipage}
\end{center}
\end{table}

From the results in Table \ref{table4}, we can observe that both \eqref{ECSDP2} and \eqref{SDPChen} are much tighter than \eqref{CSDP}. Meanwhile, the bounds of \eqref{ECSDP2} are uniformly tighter than the bounds of \eqref{SDPChen} (although the improvements are marginal). Based on the above results, we may conclude that in comparison with the previous semidefinite relaxations in the literatures, \eqref{ECSDP2} is the tightest one.

We further consider the case that $\pi< \bar{\theta}_{ij}-\underline{\theta}_{ij}<2\pi$, for which the semidefinite relaxation \eqref{SDPChen} proposed in \cite{Chen2016} can not be applied. We investigate whether \eqref{ECSDP2} can still be tighter than \eqref{CSDP} in this case. To do so, we generate another ten test instances, in which each $\underline{\theta}_{ij}$ is uniformly sampled from $(-\pi,-\pi/2)$, and $\bar{\theta}_{ij}=\underline{\theta}_{ij}+\varphi_{ij}$, where $\varphi_{ij}$ is uniformly sampled from $(\pi,2\pi)$. Then, the \eqref{ECSDP2} and \eqref{CSDP} based lower bounds are computed and listed in Table \ref{table3}.

\begin{table}[h]
\begin{center}
\begin{minipage}{\textwidth}
\caption{Computational results on instances with continuous phase angle set.}\label{table3}
\begin{tabular*}{\textwidth}{@{\extracolsep{\fill}}lrrrrrrr@{\extracolsep{\fill}}}
\toprule%
Instance & \multicolumn{2}{@{}c@{}}{Lower Bound} & \multicolumn{2}{@{}c@{}}{Time} \\ \cmidrule{2-3}\cmidrule{4-5}
ID  &\rm{CSDP} &\rm{ECSDP}  &\rm{CSDP}  &\rm{ECSDP} \\
\midrule
1 &-2984.13 &-2741.61 &0.02 &0.07  \\
2 &-3080.58 &-2798.36 &0.01 &0.04  \\
3 &-3170.00 &-2972.45 &0.01 &0.04  \\
4 &-3630.20 &-3414.64 &0.01 &0.04  \\
5 &-3062.33 &-2866.65 &0.01 &0.04  \\
6 &-3225.01 &-2945.75 &0.01 &0.04  \\
7 &-2760.60 &-2513.31 &0.01 &0.04  \\
8 &-3174.43 &-2892.12 &0.01 &0.04  \\
9 &-3032.56 &-2909.40 &0.01 &0.04  \\
10 &-3083.04 &-2827.94 &0.01 &0.04  \\
\bottomrule
\end{tabular*}
\end{minipage}
\end{center}
\end{table}

From the results in Table \ref{table3}, we can find that even for the case that $\pi< \bar{\theta}_{ij}-\underline{\theta}_{ij}<2\pi$, \eqref{ECSDP2} can be tighter than \eqref{CSDP}. These results indicate that the new semidefinite relaxation \eqref{ECSDP2} can be applied to more general cases of \eqref{CQP}, whereas the previous one in \cite{Chen2016} is specially designed for the case of $-\pi/2<\underline{\theta}_{ij}<\bar{\theta}_{ij}<\pi/2$.

\section{Conclusions}

In this paper, we propose some new semidefinite relaxations for a class of complex quadratic programming problems. The main idea behind the proposed semidefinite relaxations is that the convex hull of $(X_{ii},X_{jj},X_{ij})$ is exploited to derive valid constraints for the lifted matrix $X=xx^\dag$, which are very effective for tightening the conventional semidefinite relaxation. The main technique to derive the valid constraints is to represent the entry $X_{ij}$ in its polar coordinate form, so that the convex hull of variables in the polar coordinate representations can be derived easily. Our numerical results show that the proposed semidefinite relaxation \eqref{ECSDP2} achieves better performance than the conventional ones in terms of tightness. Besides, using the new semidefinite relaxations, some previous approximation algorithms can be improved for finding better sub-optimal solutions.

The proposed semidefinite relaxations are theoretically compared with the previous semidefinite relaxations proposed in \cite{Lu2018,Lu2020} and our proposed \eqref{ECSDP2} can be tighter than the previous ones. In particular, for the homogeneous cases that satisfy the assumptions in Theorem \ref{thm8}, the semidefinite relaxations proposed in \cite{Lu2018,Lu2020} is equivalent to \eqref{CSDP}, whereas the newly proposed \eqref{ECSDP2} is tighter than \eqref{CSDP}.
Moreover, the proposed semidefinite relaxations are also compared with the one proposed in \cite{Chen2016}. As discussed in Section 4, the semidefinite relaxation proposed in \cite{Chen2016}
is designed for the case where each $\mathcal{A}_{ij}$ is a sub-interval of $(-\pi/2,\pi/2)$. In this case, the newly proposed semidefinite relaxation \eqref{ECSDP} is equivalent to the one proposed in \cite{Chen2016}, whereas \eqref{ECSDP2} can be strictly tighter than the one in \cite{Chen2016}. Moreover, \eqref{ECSDP} and \eqref{ECSDP2} can be applied to general cases of \eqref{CQP}, and keep their effectiveness even for cases where $\mathcal{A}_{ij}=[\underline{\theta}_{ij},\bar{\theta}_{ij}]$ with $\pi<\bar{\theta}_{ij}-\underline{\theta}_{ij}<2\pi$.

Based on the results presented in this paper, there are two potential directions that deserve further research. First, for the discrete case, such as the radar waveform design problem discussed in Section 5.1, whether the theoretical approximation ratio of \eqref{ECSDP2} based approximation algorithm can be better than the ratio of \eqref{CSDP} based algorithm is an interesting question to answer. Second, since \eqref{ECSDP2} can be tighter than the semidefinite relaxation proposed in \cite{Chen2016}, we may try to design a new branch-and-bound algorithm to solve the optimal power flow problem, using \eqref{ECSDP2} as a relaxation method. It is interesting to design a new branch-and-bound algorithm to compare with the one in \cite{Chen2016}.


%
%



\end{document}